\documentclass[reqno]{amsart} 
\usepackage{amsfonts, amsmath, amsthm, amssymb, latexsym, stmaryrd, graphicx}

\theoremstyle{plain}
\newtheorem{theorem}{Theorem}[section]
\newtheorem{corollary}[theorem]{Corollary}
\newtheorem{lemma}[theorem]{Lemma}
\newtheorem{proposition}[theorem]{Proposition}
\theoremstyle{definition}

\newtheorem{remark}[theorem]{Remark}
\theoremstyle{remark}

\numberwithin{equation}{section}

\title[Bessel processes in the freezing regime]{Limit theorems
 for multivariate Bessel processes in the freezing regime}
\author{Sergio Andraus}
\address{Faculty of Science and Engineering, Chuo University, Kasuga 1-13-27, Bunkyo-Ku, Tokyo 112-8551, Japan}
\email{andraus@phys.chuo-u.ac.jp}
\author{Michael Voit}
\address{Fakult\"at Mathematik, Technische Universit\"at Dortmund,
          Vogelpothsweg 87,
          D-44221 Dortmund, Germany}
\email{michael.voit@math.tu-dortmund.de}


\begin{document}
\date{\today}

\begin{abstract} 
Multivariate Bessel processes describe the stochastic dynamics of interacting particle systems of 
Calogero-Moser-Sutherland type and are related with $\beta$-Hermite and Laguerre ensembles. 
It was shown by Andraus, Katori, and Miyashita that for fixed starting points, these processes admit
interesting limit laws when the  multiplicities $k$ tend to $\infty$, where in some cases
the limits are described by the zeros of classical Hermite and Laguerre 
polynomials. In this paper we use SDEs to derive corresponding limit laws for
 starting points of the form $\sqrt k\cdot x$ for $k\to\infty$ with $x$ in the interior of the corresponding Weyl chambers.
Our limit results are a.s. locally uniform in time. Moreover,
in some cases we present associated central limit theorems.
\end{abstract}

\maketitle

\section{Introduction} 

The dynamics of integrable interacting particle systems of 
Calogero-Moser-Suther\-land type on the real line $\mathbb R$ with $N$ particles 
 can be described by certain time-homogeneous diffusion processes on 
suitable closed subsets of $\mathbb R^N$. These processes are often called (multivariate) 
Bessel, Dunkl-Bessel, or radial Dunkl processes; for their detailed definition and properties, see \cite{CGY,GY,R1,R2, 
RV1,RV2,DV,A}.
These processes are classified via root systems and  finitely many  multiplicity parameters
which act as coupling constants of interaction.
In this paper, we restrict our attention  to the root systems that are mainly important for particle systems and 
in random matrix theory, namely
those of the types $A_{N-1}$, $B_N$, and $D_N$, as here the number $N$ of particles is arbitrary. 
Besides these root systems and a finite number of exceptional cases,
 there are the dihedral
sytems with $N=2$ (see \cite{Dem2}) as well direct products. 
We shall not study these cases in this paper.

To explain the results of this paper, we briefly recapitulate some well-known basic facts.
In the  case $A_{N-1}$, we have a one-dimensional multiplicity $k>0$, the processes live on the closed Weyl chamber
\[C_N^A:=\{x\in \mathbb R^N: \quad x_1\ge x_2\ge\ldots\ge x_N\},\]
 the generator of the transition semigroup is given by
\begin{equation}\label{def-L-A} Lf:= \frac{1}{2} \Delta f +
 k \sum_{i=1}^N\Bigl( \sum_{j\ne i} \frac{1}{x_i-x_j}\Bigr) \frac{\partial}{\partial x_i}f ,
 \end{equation}
and we assume reflecting boundaries. The latter means that the domain of $L$ may be chosen as
\[D(L):=\{f|_{C_N^A}: \>\> f\in C^{(2)}(\mathbb R^N), \>\>\> f\>\>\text{ invariant under all coordinate permutations}\}.\]

In the case $B_N$,  we have 2  multiplicities $k_1,k_2\ge 0$, the processes live on 
\[C_N^B:=\{x\in \mathbb R^N: \quad x_1\ge x_2\ge\ldots\ge x_N\ge0\},\]
 the generator of the transition semigroup is 
\begin{equation}\label{def-L-B} Lf:= \frac{1}{2} \Delta f +
 k_2 \sum_{i=1}^N \sum_{j\ne i} \Bigl( \frac{1}{x_i-x_j}+\frac{1}{x_i+x_j}  \Bigr)
 \frac{\partial}{\partial x_i}f 
\quad + k_1\sum_{i=1}^N\frac{1}{x_i}\frac{\partial}{\partial x_i}f, \end{equation}
and we again assume reflecting boundaries, i.e., the domain of $L$ is 
\begin{align}D(L):=\{f|_{C_N^B}:&  \>\>f\in C^{(2)}(\mathbb R^N),
 \>\>\> f\>\>\text{ invariant under all permutations of } \notag\\
&\text{ coordinates and under all sign changes in all coordinates}\}.\notag
\noindent\end{align}

We study limit theorems for these diffusions $(X_{t,k})_{t\ge0}$ on $C_N$ (with $C_N=C_N^A$ or $C_N^B$)
for the fixed times $t>0$ in freezing regimes, where $k$ stands for $k\ge0$ in the $A_{N-1}$-case, and for $(k_1,k_2)$
in the $B_N$-case. Freezing means that for fixed times $t>0$, we consider
 $k\to\infty$ in the $A_{N-1}$-case, and in the $B_N$-case, the two cases  $(k_1,k_2)=(\nu\cdot \beta,\beta)$
with $\nu>0$ fixed, $\beta\to\infty$ as well as   $k_2>0$ fixed, $k_1\to \infty$. 
For these limit cases,
\cite{AKM1,AKM2,AM} present weak limit laws for $X_{t,k}$ for  fixed times $t>0$ when the processes
 start in the origin $0\in C_N$  or with a fixed starting distribution independent from $k$.
In this paper we shall derive similar limit results when the starting points of the diffusions  $(X_{t,k})_{t\ge0}$ 
depend on $k$, more precisely, if the  starting points have the form $\sqrt\kappa\cdot x$
 where $\kappa$ is the parameter in the coupling constants which tends to $\infty$, and where $x$ is a point in the interior of 
$C_N$. The last condition will be essential in this paper, as we shall apply an SDE 
approach to the limit results which works properly only in the interior of $C_N$ as the SDEs become singular on the boundaries.
It will turn out on an informal level that the limit results in \cite{AKM1,AKM2,AM} 
 may be seen as special cases of our results for $x=0$, even if the case $x=0$ is not covered by our approach.

To explain the connection of our results with \cite{AKM1,AKM2,AM}, 
we  recapitulate some further details.
The transition probabilities of the Bessel processes are given for all root systems as
 follows by \cite{R1,R2,RV1,RV2}:
For $t>0$,  $x\in C_N$, $A\subset C_N$ a Borel set, 
\begin{equation}\label{density-general}
K_t(x,A)=c_k \int_A \frac{1}{t^{\gamma+N/2}} e^{-(\|x\|^2+\|y\|^2)/(2t)} J_k\Big(\frac{x}{\sqrt{t}}, \frac{y}{\sqrt{t}}\Big) 
\cdot w_k(y)\> dy
\end{equation}
with
\begin{equation}\label{def-wk}
w_k^A(x):= \prod_{i<j}(x_i-x_j)^{2k}, \quad\quad 
w_k^B(x):= \prod_{i<j}(x_i^2-x_j^2)^{2k_2}\cdot \prod_{i=1}^N x_i^{2k_1},\end{equation}
and
\begin{equation}\label{def-gamma}
\gamma_A(k)=kN(N-1)/2, \quad\quad  \gamma_B(k_1,k_2)=k_2N(N-1)+k_1N
\end{equation}
respectively. 
$w_k$ is homogeneous of degree $2\gamma$.
Furthermore,
$c_k>0$ is a known normalization constant, and
$J_k$ is a multivariate Bessel function of type $A_{N-1}$ or $B_N$ with multiplicities $k$ or $(k_1,k_2)$ respectively;
 see e.g. \cite{R1,R2}.

We do not need much information about $J_k$. We only recapitulate that 
$J_k$ is analytic on $\mathbb C^N \times \mathbb C^N $ with
$ J_k(x,y)>0$ for $x,y\in \mathbb R^N $.
Moreover, $J_k(x,y)=J_k(y,x)$ and $J_k(0,y)=1$
for all $x,y\in \mathbb C^N $.

Therefore, if we start the process from $0$, then $X_{t,k}$ has the Lebesgue density
\begin{equation}\label{density-A-0}
 \frac{c_k}{t^{\gamma+N/2}} e^{-\|y\|^2/(2t)} \cdot w_k(y)\> dy
\end{equation}
on $C_N$ for $t>0$. In the case $A_{N-1}$, the density of
 $X_{t,k}/{\sqrt{tk}}$ has the form
\[\textrm{const.}(k)\cdot \exp\Bigl( k\Bigl(2\sum_{i,j: i<j} \ln(y_i-y_j) -\|y\|^2/2\Bigr)\Bigr)=:
 \textrm{const.}(k)\cdot \exp\Bigl( k\cdot W(y)\Bigr)\]
which is  well-known for $k=1/2, 1,2$ as the distribution of the eigenvalues of
Gaussian orthogonal, unitary, and symplectic ensembles; see e.g. \cite{D}.
Moreover, for general $k>0$, (\ref{density-A-0}) appears as the distribution of the tridiagonal matrix models of
Dumitriu and Edelman \cite{DE1,DE2} for $\beta$-Hermite and $\beta$-Laguerre ensembles.

In the case $A_{N-1}$, 
it was observed in  \cite{AKM1} (see also Section 6.7 of \cite{S}) that the maximum of $W$ on $C_N^A$ appears precisely
 for $y=\sqrt 2\cdot z$ where $ z\in C_N^A$ is the vector
whose entries are the zeroes of the classical  Hermite polynomial $H_N$
where  the $(H_N)_{N\ge 0}$ are orthogonal w.r.t.
 the density  $e^{-x^2}$.
This shows that 
$X_{t,k}/\sqrt{2tk}$ tends to $z$ in distribution for $k\to\infty$.
This means that 
\begin{equation}\label{LLN-A-start-0}
\lim_{k\to\infty}\frac{X_{t,k}}{\sqrt{2tk}}= z
\end{equation}
in probability whenever the $X_{t,k}$ are defined on a common probability space. In fact, this result was proved in 
\cite{AKM1} in a  more general form,
 namely for arbitrary fixed starting distributions.
Moreover, (\ref{LLN-A-start-0}) and an associated central limit theorem for start in $0$ was derived in \cite{DE2}
via the explicit  tridiagonal matrix model of
Dumitriu and Edelman \cite{DE1}; see also \cite{V2} for another elementary approach.

We now compare  (\ref{LLN-A-start-0}) with the main results here for the case $A_{N-1}$.
We show in Theorem \ref{SLLN-A}  below that the Bessel processes $(X_{t,k})_{t\ge0}$ with start in $\sqrt k\cdot x$ (for some 
point $x$ in the interior of $C_N$) satisfy
\begin{equation}\label{LLN-A-start-interior-introduction}
X_{t,k}/\sqrt k\to \phi(t,x) \quad\quad\text{for}\quad\quad k\to\infty
\end{equation}
with an error of size $O(1/\sqrt k)$ locally uniformly in $t$ almost surely where 
$\phi(t,x)$ is the solution of a (deterministic) 
 dynamical system at time $t>0$, where the system starts at time $0$ in $x$. For the details we refer to Section 2.

For the root systems $B_N$ and $D_N$ we shall derive corresponding results.

We mention that the  locally uniform convergence in $t$ in (\ref{LLN-A-start-interior-introduction})
 and the corresponding results for the other root systems we discuss ensures that we can interchange limits for $k\to\infty$
with stochastic integrals. This finally leads, in combination with the SDEs,
 to central limit theorems (CLTs) at least in some cases; see Section 4 below for some cases and also \cite{VW}
for other cases.

The basis for the SDE-approach for all root systems
is the following well-known result (see Lemma 3.4, Corollary 6.6, and Proposition 6.8 of \cite{CGY}):

\begin{theorem}\label{SDE-basic} Let $k>0$ in the 
$A_{N-1}$-case or $k_1,k_2>0$ in the $B_N$-case. Then, for each starting point $x\in C_N$ and  $t>0$,
the Bessel process $(X_{t,k})_{t\ge0}$ satisfies
 \[E\Bigl(\int_0^t \nabla(\ln w_k)(X_{s,k}) \> ds\Bigr)<\infty.\]

 Moreover, the initial value problem
\begin{equation}\label{SDE-general}
X_0=x,\quad\quad dX_t= dB_t +  \frac{1}{2} (\nabla(\ln w_k))(X_t) \> dt
\end{equation}
 (with an $N$-dimensional Brownian motion $(B_t)_{t\ge0}$) 
has a unique (strong) solution $(X_t)_{t\ge0}$.
This solution is a Bessel process as  above.

Moreover, if  $k\ge 1/2$ in the 
$A_{N-1}$-case or $k_1,k_2\ge 1/2$ in the $B_N$-case, and if
 $x$ is in the interior of $C_N$, then $(X_t)_{t\ge0}$ lives on the interior on $C_N$,
 i.e. the solution does not meet the boundary almost surely.
\end{theorem}

This paper is organized as follows. In the next section we derive a strong limit law (LL) for 
Bessel processes of type $A_{N-1}$ for $k\to\infty$.
Section 3 is then devoted to corresponding LLs in the case $B_N$ for two freezing regimes which were already studied in 
\cite{AKM2,AM}. For the regime $k_1\to\infty$ and $k_2>0$ fixed, we use the locally uniform LL in Section 4,
 in order to derive an associated CLT. Finally, in Section 5 we consider the LL for the root systems
of type $D_N$, which are related with the $B_N$-case for  $k_2=0$.

\section{Strong limiting law for the root system $A_{N-1}$}

In this section we derive a locally uniform strong LL in the case  $A_{N-1}$ for $k\to\infty$ for starting points 
of the form $\sqrt k\cdot x$ with $x$ in the interior of $C_N^A$. 
The main result corresponds to the weak LL  (\ref{LLN-A-start-0}) started at the origin.

For  $k\ge1/2$ and any starting point $x=x(k)$ in the interior of $C_N^A$ we study
 the associated  Bessel process $(X_{t,k})_{t\ge0}$ of type $A_{N-1}$ which may be seen as
 the   unique solution of the initial value problem
(\ref{SDE-general}). In the $A_{N-1}$ case the SDE reads
\begin{equation}\label{SDE-A}
 dX_{t,k}^i = dB_t^i+ k\sum_{j\ne i} \frac{1}{X_{t,k}^i-X_{t,k}^j}dt \quad\quad(i=1,\ldots,N),
\end{equation}
with an $N$-dimensional Brownian motion $(B_t^1,\ldots,B_t^N)_{t\ge0}$.
In order to derive LLs for $X_{t,k}$,
we study the renormalized processes $(\tilde X_{t,k}:=X_{t,k}/\sqrt k)_{t\ge0}$ which satisfy
\begin{equation}\label{SDE-A-normalized}
d\tilde X_{t,k}^i =\frac{1}{\sqrt k}dB_t^i + \sum_{j\ne i} 
 \frac{1}{\tilde X_{t,k}^i-\tilde X_{t,k}^j}dt\quad\quad(i=1,\ldots,N). 
\end{equation}
We  compare (\ref{SDE-A-normalized}) with the deterministic limit case $k=\infty$. This limit case
has the following properties:

\begin{lemma}\label{deterministic-boundary-A}
For $\epsilon>0$ consider the open subset $U_\epsilon:=\{x\in C_N^A:\> d(x,\partial C_N^A)>\epsilon\}$
 (where $\mathbb R^N$ carries the usual Euclidean norm and $d(x,y)$ denotes the distance).
Then the function 
\[H:U_\epsilon\to \mathbb R^N, \quad x\mapsto \Bigl( \sum_{j\ne1} \frac{1}{x_1-x_j},\ldots,\sum_{j\ne N} \frac{1}{x_N-x_j}
\Bigr)\]
is Lipschitz continuous on $U_\epsilon$ with Lipschitz constant $L_\epsilon>0$,
 and for each starting point $x_0\in U_\epsilon$, the solution $\phi(t,x_0)$
of the dynamical system $\frac{d}{dt}x(t) =H(x(t))$ satisfies $\phi(t,x_0)\in U_\epsilon$ for all $t\ge0$.
\end{lemma}

\begin{proof} For $x\in U_\epsilon$ and $i\ne j$ we have $|x_i-x_j|>\epsilon$. Hence there is a constant $C>0$ with
$|\frac{\partial}{\partial x_i} H(x)|\le C$ for $x\in U_\epsilon$ and $i=1,\ldots, N$
 which implies the Lipschitz continuity.

For the second statement we use the new variables $y_i(t):=x_i(t)-x_{i+1}(t)> \epsilon$ for $x\in U_\epsilon$ and
 $i=1,\ldots, N-1$. Then
\begin{align}
\frac{dy_i}{dt}(t) =& \frac{1}{y_1+\cdots+y_i}+\frac{1}{y_i+\cdots+y_{N-1}}\notag\\
&+\sum_{j=1}^{i-1}  \Bigl( \frac{1}{y_{j+1}+\cdots+ y_{i}}-\frac{1}{y_j+\cdots+ y_{i-1}}\Bigr)\notag\\
&+\sum_{j=i+2}^{N} \Bigl( \frac{1}{y_{i}+\cdots+ y_{j-2}} - \frac{1}{y_{i+1}+\cdots+ y_{j-1}}\Bigr).\notag
\end{align}
For any $t\ge0$ choose some $i=i(t)$ for which $y_i(t)$ is minimal, i.e., $y_j(t)\ge y_i(t)$ for all $j$.
Notice that  $i=i(t)$ is not necessarily unique. However, for each $i=i(t)$ of this kind
 we have
\[\frac{1}{y_{i+1}(t)}\le \frac{1}{y_{i}(t)} \quad\quad\text{and}\quad\quad
\frac{1}{y_{i}(t)+\cdots+ y_{j-2}(t)} \ge \frac{1}{y_{i+1}(t)+\cdots+ y_{j-1}(t)} \quad(j=i+3,\ldots, N)\]
and
\[\frac{1}{y_{i-1}(t)}\le \frac{1}{y_{i}(t)} \quad\quad\text{and}\quad\quad
\frac{1}{y_{j+1}(t)+\cdots+ y_{i}(t)} \ge \frac{1}{y_{j}(t)+\cdots+ y_{i-1}(t)} \quad(j=1,\ldots, i-2).\]
Therefore,
\[\frac{dy_i}{dt}(t) \ge \frac{1}{y_{1}(t)+\cdots+ y_{i}(t)}+ \frac{1}{y_{i}(t)+\cdots+ y_{N-1}(t)}>0.\]
Hence, for each $t\ge0$ there is a neighborhood on which $y_i$ is increasing for each $i$, for which 
 $y_i(t)$ is minimal. This means 
 that $s\mapsto\min_{i=1,\ldots,N-1}y_i(s)$ is increasing in this neighborhood of $t$. This completes the proof of the lemma.
\end{proof}

It seems that the dynamical system from Lemma \ref{deterministic-boundary-A} can be solved explicitly
 only for a few cases like $N=2$ or particular starting points which are related to the zeros of the Hermite polynomial 
$H_N$. The latter is not surprising in  view of the LLs of \cite{AKM1}. 
To explain these solutions, we recall the following fact (see \cite{AKM1} and  Section 6.7 of \cite{S}):

\begin{lemma}\label{char-zero-A}
 For $y\in C_N^A$, the following statements  are equivalent:
\begin{enumerate}
\item[\rm{(1)}] The function $W(x):=2\sum_{i,j: i<j} \ln(x_i-x_j) -\|x\|^2/2$ is maximal at $y\in C_N^A$;
\item[\rm{(2)}] For $i=1,\ldots,N$:  $\frac{1}{2}y_i= \sum_{j: j\ne i} \frac{1}{y_i-y_j}$;
\item[\rm{(3)}] The vector 
\[z:=(z_1,\ldots,z_N):=(y_1/\sqrt2, \ldots,y_N/\sqrt2)\]
 consists of
 the ordered zeroes of the classical Hermite polynomial $H_N$.
\end{enumerate}
\end{lemma}

Part (3) of this lemma immediately leads to the following solution of the differential equation of
 Lemma \ref{deterministic-boundary-A}:

\begin{corollary}
For each $c>0$, a particular solution of the dynamical system in  Lemma \ref{deterministic-boundary-A}
is given by $\phi(t,c\cdot {z})= \sqrt{2t+c^2}\cdot {z} $.
\end{corollary}

Notice that on an informal level the same statement holds also for $c=0$.

We now turn to the main result of this section, a locally uniform strong LL with a strong  order of convergence:

\begin{theorem}\label{SLLN-A} Let $x$ be a point in the interior of $C_N^A$, and let $y\in \mathbb R^N$. 
Let $k_0\ge 1/2$ with $\sqrt k \cdot x+y$ in the interior of $C_N^A$ for $k\ge k_0$.

For $k\ge k_0$ consider the Bessel processes $(X_{t,k})_{t\ge0}$ of type $A_{N-1}$ on   $C_N^A$, started at
$\sqrt k\cdot x+y$, and which satisfy the SDEs
\[ dX_{t,k}^i = dB_t^i+ k\sum_{j\ne i} \frac{1}{X_{t,k}^i-X_{t,k}^j}dt \quad\quad(i=1,\ldots,N).\]
Then, for all $t>0$,
\[\sup_{0\le s\le t, k\geq k_0}\|X_{s,k}-\sqrt k \phi(s,x) \|<\infty\]
 almost surely.
In particular,
\[X_{t,k}/\sqrt k\to \phi(t,x) \quad\quad\text{for}\quad\quad k\to\infty\]
locally uniformly in $t$ almost surely and thus locally uniformly in $t$ in probability.
\end{theorem}

\begin{proof}
Recall that the processes $(\tilde X_{t,k}:=X_{t,k}/\sqrt k)_{t\ge0}$ satisfy
\[\tilde X_{t,k}^i =\frac{1}{\sqrt k}(y_i+B_t^i) +x_i + \int_0^t\sum_{j\ne i} 
 \frac{1}{\tilde X_{s,k}^i-\tilde X_{s,k}^j}ds\quad\quad(i=1,\ldots,N). \]
We compare the solutions of these SDEs with the solution $Y_t=\phi(t,x)$ ($t\ge0$) of the deterministic equation
\[Y_t^i =x_i + \int_0^t\sum_{j\ne i} \frac{1}{Y_{s}^i-Y_{s}^j}ds\quad\quad (i=1,\ldots,N) \]
of Lemma \ref{deterministic-boundary-A}.
For both equations we perform Picard iterations as follows. We set the starting points at
\[\tilde  X_{t,k,0}:= Y_{t,0}:=x\]
and, for $m\ge0$, we set the recursions
\[\tilde  X_{t,k,m+1}^i:= \frac{1}{\sqrt k}(y_i+B_t^i) +x_i + \int_0^t\sum_{j\ne i} 
\frac{1}{\tilde X_{s,k,m}^i-\tilde X_{s,k,m}^j}ds  \quad\quad(i=1,\ldots,N)\]
and
\[Y_{t,m+1}^i := x_i + \int_0^t\sum_{j\ne i} \frac{1}{Y_{s,m}^i-Y_{s,m}^j}ds
\quad\quad(i=1,\ldots,N).\]

For given points $x,y$ and given $k_0$ as in the statement, we find
 $\epsilon>0$ small enough
 that $x+y/\sqrt k\in U_\epsilon$ for $k\ge k_0$ 
where  $U_\epsilon$ is given as in Lemma \ref{deterministic-boundary-A}.
Consider the stopping times 
\[T_{\epsilon,k}:=\inf\{t>0:\> \tilde  X_{t,k}\not\in U_\epsilon\}.\]
We  study the stopped maximal differences
\[ D_{t,k,m,\epsilon}:=\sup_{s\in [0,t\wedge T_{\epsilon,k}]}\| \tilde  X_{s,k,m}-Y_{s,m}\| 
\quad\quad(t\ge0, m\ge0)\]
with $D_{t,k,0,\epsilon}=0$. Using the Lipschitz constants $L_\epsilon>0$  on 
$U_\epsilon$ as in  Lemma \ref{deterministic-boundary-A}, we obtain
\begin{equation*}
D_{t,k,m+1,\epsilon} \le \frac{1}{\sqrt k}(\|y\|+ \sup_{s\in [0,t]}\|B_s\| ) +
 L_\epsilon\cdot \sup_{s\in [0,t]} \int_0^s D_{u,k,m,\epsilon}\> du. \notag
\end{equation*}
Induction on $m$ shows that for all $m$,
\begin{align}\label{est-approximations-A}
D_{t,k,m,\epsilon} &\le\frac{1}{\sqrt k}(\|y\|+ \sup_{s\in [0,t]}\|B_s\| )\cdot\Bigl( \sum_{l=0}^{m-1} 
(L_\epsilon t)^l \frac{1}{l!}\Bigr) \notag\\
 &\le\frac{1}{\sqrt k}(\|y\|+ \sup_{s\in [0,t]}\|B_s\| )\cdot e^{L_\epsilon t}.
\end{align}
On the other hand, it is well known from the classical theory of SDEs 
(see, for example, Theorems 7 and 8 of Section V.3 of 
\cite{P}) that under a Lipschitz condition, for
$m\to\infty$,
\[\sup_{s\in [0,t\wedge T_{\epsilon,k}]}\|\tilde  X_{s,k,m}-\tilde  X_{s,k}\| \to0
\quad\text{and}\quad \sup_{s\in [0,t\wedge T_{\epsilon,k}]}\| Y_{s,m}-  Y_{s}\| \to0\]
in probability. This means that some subsequence  converges almost surely.
We conclude from (\ref{est-approximations-A}) that
\begin{equation}\label{est-limit-A}
\sup_{s\in [0,t\wedge T_{\epsilon,k}]}\|\tilde X_{s,k}-Y_{s}\|\le \frac{1}{\sqrt{k}}\cdot C_t\cdot  e^{L_\epsilon t}
\end{equation}
almost surely where $C_t:=\|y\|+ \sup_{s\in [0,t]}\|B_s\| $ is a random variable
which is almost surely finite. 

We now consider the events $\Omega_M:=\{\omega:\> C_t(\omega)\le M\}$ for $M\in\mathbb N$
which satisfy $P(\Omega_M)\to1$ for $M\to\infty$. For given $x,t,M$ and $\epsilon$ we enlarge
 $k_0:=k_0(x,y,t,M,\epsilon)$ such that in addition,
\[\frac{1}{\sqrt k_0}M\cdot  e^{L_\epsilon t}< \frac{d(x,\partial C_N^A)-\epsilon}{2}.\]
Then, for $k\ge k_0$ and $\omega\in \Omega_M$,
\[\sup_{s\in [0,t\wedge T_{\epsilon,k}(\omega)]}\|\tilde X_{s,k}(\omega)-Y_{s}\|< \frac{d(x,\partial C_N^A)-\epsilon}{2}.\]
As $d(Y_s, \partial C_N^A)\ge d(x, \partial C_N^A)$ for $s\ge0$ 
by Lemma \ref{deterministic-boundary-A}, we see that for $s\in [0,t\wedge T_{\epsilon,k}(\omega)] $,
 \begin{equation}\label{boundary-away-A}
d(\tilde X_{s,k}(\omega),\partial C_N^A)\ge \frac{d(x,\partial C_N^A)+\epsilon}{2}> \epsilon.
\end{equation}
Because the paths of the Bessel processes we consider are almost surely continuous, we conclude that $T_{\epsilon,k}(\omega)=\infty$ and thus
$\tilde X_{s,k}(\omega)\in U_\epsilon$ for all $s\in [0,t]$,
 $\omega\in \Omega_M$ and $k\ge k_0$. Hence, for  $\omega\in \Omega_M$ and $k\ge k_0$,
\begin{equation}\label{uniform-final}
\sup_{s\in [0,t]}\|\tilde X_{s,k}(\omega)-Y_{s}\|\le \frac{1}{\sqrt k_0}M\cdot  e^{L_\epsilon t}.
\end{equation}
As $P(\Omega_M)\to1$ for $M\to\infty$,  the first statement of the theorem is clear, and
the second statement follows immediately from taking the limit $k_0\to\infty$, which forces $k\to\infty$.
\end{proof}

\begin{remark} Theorem \ref{SLLN-A} can be easily generalized to the case where the points $x,y\in\mathbb R^N$
are independent random variables $X,Y$ which are also independent of the Brownian motion $(B_t)_{t\ge0}$. 

In fact, if $X$ has values in an open subset $U_\epsilon$ of $C_N^A$ as described
 in Lemma \ref{deterministic-boundary-A}, and if there exists $k_0>0$ such that $\sqrt k X+Y$ has values in 
 $C_N^A$ for $k\ge k_0$, then the proof of Theorem \ref{SLLN-A}  still holds, that is, we obtain
 that the Bessel processes
$( X_{t,k})_{t\ge0}$ with $X_{0,k}=\sqrt k X +Y$ satisfy
\begin{equation}\label{uniform-final-random}
sup_{s\in [0,t]}\| X_{s,k}(\omega)-\sqrt k\cdot \phi(s,X)\| <\infty \quad\quad\text{a.s..}
\end{equation}

Moreover, if ${\bf P}(X\in\partial C_N^A)=0$, and if  $\sqrt k X+Y$ has values in 
 $C_N^A$ for $k\ge k_0$ given $k_0>0$, then 
\[k^\alpha ( X_{t,k}/\sqrt k -\phi(t,X))\to0\]
for all $\alpha<1/2$ in probability. This also follows immediately from Theorem \ref{SLLN-A}
and the fact that ${\bf P}(X\in\partial C_N^A)=0$ implies that  ${\bf P}(X\in U_{1/n})\to1$
for $n\to\infty$.

We also remark that the limiting laws \ref{SLLN-B1}, \ref{SLLN-B2}, and \ref{SLLN-D}
 below for the root systems $B_N$ and $D_N$ and fixed starting points can be also extended to
 random starting points in the same way.
\end{remark}

\section{Strong limiting laws for the root system $B_{N}$}

In this section we derive LLs in the case  $B_{N}$ for the two freezing regimes
from the introduction for starting points 
in the interior of  $C_N^B$. 
In both cases we  consider $k=(k_1,k_2)$ with $k_1,k_2>0$, and study
 Bessel process $(X_{t,k})_{t\ge0}$ which are solutions of (\ref{SDE-general}). In the B-case, 
 the SDE (\ref{SDE-general})
 reads
\begin{equation}\label{SDE-B}
 dX_{t,k}^i = dB_t^i+ k_2\sum_{j\ne i} \Bigl(\frac{1}{X_{t,k}^i-X_{t,k}^j} +   \frac{1}{X_{t,k}^i+X_{t,k}^j}\Bigr) dt +
 \frac{k_1}{X_{t,k}^i} dt
\end{equation}
for $i=1,\ldots,N$
with an $N$-dimensional Brownian motion $(B_t^1,\ldots,B_t^N)_{t\ge0}$.

The two freezing regimes 
have to be handled differently from the previous, $A_{N-1}$, case.
We start with the case
 $(k_1,k_2)=(\nu\cdot \beta,\beta)$
with $\nu>0$ fixed and $\beta\to\infty$ which was studied in \cite{AKM2, AM} for the case of a fixed
starting distribution on $C_N^B$.
Similar to the $A_{N-1}$ case, 
we study the renormalized processes $(\tilde X_{t,k}:=X_{t,k}/\sqrt {\beta})_{t\ge0}$ which satisfy
\begin{equation}\label{SDE-B1-normalized}
d\tilde X_{t,k}^i =\frac{1}{\sqrt \beta}dB_t^i + \sum_{j\ne i} 
\Bigl( \frac{1}{\tilde X_{t,k}^i-\tilde X_{t,k}^j} +   \frac{1}{\tilde X_{t,k}^i+\tilde X_{t,k}^j}\Bigr)dt
+
 \frac{\nu}{\tilde X_{t,k}^i} dt \end{equation}
for $i=1,\ldots,N$.
We again compare   $\tilde X_{t,k}$ with the solution of a deterministic dynamical system.

\begin{lemma}\label{deterministic-boundary-B1}
Let $\nu>0$. For $\epsilon>0$ consider the open subset
 \[U_\epsilon:=\{x\in C_N^B:\> x_N>\frac{\epsilon \nu}{N-1}, \quad\text{and}\quad
x_i-x_{i+1}>\epsilon \quad\text{for}\quad i=1,\ldots,N-1  \}.\]
Then $\cup_{\epsilon>0}U_\epsilon$ is the interior of $C_N^B$, and
 the function 
\[H:U_\epsilon\to \mathbb R^N, \quad x\mapsto 
\left(\begin{matrix}\sum_{j\ne1}\Bigl( \frac{1}{x_1-x_j}+   \frac{1}{x_1+x_j}\Bigr)+
\frac{\nu}{x_1}\\
\vdots\\
\sum_{j\ne N} \Bigl(\frac{1}{x_N-x_j}+   \frac{1}{x_N+x_j}\Bigr)+
\frac{\nu}{x_N}
\end{matrix}\right)\]
is Lipschitz continuous on $U_\epsilon$ with Lipschitz constant $L_\epsilon>0$. Moreover,
 for each starting point $x_0\in U_\epsilon$, the solution $\phi(t,x_0)$
of the dynamical system $\frac{dx}{dt}(t) =H(x(t))$ satisfies $\phi(t,x_0)\in U_\epsilon$ for all $t\ge0$.
\end{lemma}

\begin{proof} There exits a constant $\tilde \epsilon>0$ such that for all
 $x\in U_\epsilon$ and $i\ne j$ we have $|x_i\pm x_j|>\tilde\epsilon$ and $x_i>\tilde\epsilon$.
 Hence there is a constant $C>0$ with
$|\frac{\partial H}{\partial x_i}(x)|\le C$ for $x\in U_\epsilon$ and $i=1,\ldots, N$
 which implies the Lipschitz continuity.

We now proceed as in the proof of Lemma \ref{deterministic-boundary-B1}
and use the new variables
 $y_i(t):=x_i(t)-x_{i+1}(t)> \epsilon$ for  $i=1,\ldots, N-1$ as well as $y_N(t):= \frac{N-1}{\nu}\cdot x_N(t)$. 
For any $t\ge0$ we choose $i=i(t)$ for which $y_i(t)$ is minimal, i.e., $y_j(t)\ge y_i(t)$ for all $j$.
If  $i\in\{1,\ldots,N-1\}$, then the estimations in the proof of  Lemma \ref{deterministic-boundary-B1}
immediately imply  $\frac{dy_i}{dt}(t)\ge0$, as the right hand side of the dynamical system here 
is clearly greater than the right hand side of the  system in Lemma \ref{deterministic-boundary-A}.

Moreover, for $i=N$ (that is, if $y_N(t)=\min_{1\leq j\leq N}y_j(t)$),
\begin{align}
\frac{dy_N}{dt}(t)&=\frac{N-1}{\nu} \Bigl[\sum_{j\ne N} \Bigl(\frac{1}{x_N(t)-x_j(t)}+   \frac{1}{x_N(t)+x_j(t)}\Bigr)+
\frac{\nu}{x_N(t)} \Bigr]\notag\\
&\ge \frac{N-1}{\nu} \Bigl[\frac{\nu}{x_N(t)} - \frac{N-1}{x_{N-1}-x_N(t)}\Bigr] \ =
\ \frac{(N-1)^2}{\nu} \Bigl( \frac{1}{y_N(t)} - \frac{1}{y_{N-1}(t)} \Bigr)
\ge0.\notag
\end{align}
In summary, we see that $\min_{i=1,\ldots,N}y_i(t)$ is increasing in $t$. This completes the proof.
\end{proof}

As in the $A_{N-1}$ case, it seems difficult to solve  the dynamical systems of Lemma \ref{deterministic-boundary-B1} except
 for a few cases like $N=1$ or particular starting points which are related to the zeros of certain Laguerre polynomials.
 The latter is not surprising in  view of the LLs of \cite{AKM2}. 
To explain this, we recapitulate the following fact; see \cite{AKM2} and  Section 6.7 of \cite{S} and notice that our
parameters $(\beta,\nu)$ correspond to the parameters $(\beta/2, \nu+1/2)$ in  \cite{AKM2}:

\begin{lemma}\label{char-zero-B1}
Let $\nu>0$. For $y\in C_N^B$, the following statements  are equivalent:
\begin{enumerate}
\item[\rm{(1)}] The function 
\[W(x):=2\sum_{ i<j} \ln(x_i^2-x_j^2) +2\nu \sum_{i}\ln x_i-\|x\|^2/2\]
 is maximal at $y\in C_N^B$;
\item[\rm{(2)}] For $i=1,\ldots,N$, 
\[\frac{1}{2}y_i= \sum_{j: j\ne i} \frac{2y_i}{y_i^2-y_j^2} +\frac{\nu}{y_i}=\sum_{j: j\ne i} \Bigl(
\frac{1}{y_i-y_j} +\frac{1}{y_i+y_j}\Bigr) +\frac{\nu}{y_i} ;\] 
\item[\rm{(3)}] If $z_1^{(\nu-1)},\ldots,z_N^{(\nu-1)}$ are the ordered zeros of 
 the classical  Laguerre polynomial $L_N^{(\nu-1)}$ (where the  $L_N^{(\nu-1)}$ are
orthogonal w.r.t. the density $e^{-x}\cdot x^{\nu-1}$), then 
\begin{equation}\label{y-max-B1}
2(z_1^{(\nu-1)},\ldots, z_N^{(\nu-1)})= (y_1^2, \ldots, y_N^2).
\end{equation}
\end{enumerate}
\end{lemma}

\begin{remark}\label{remark-b-nu0}
Using the known explicit representation
\[L_N^{(\alpha)}(x):=\sum_{k=0}^N { N+\alpha\choose N-k}\frac{(-x)^k}{k!}\]
of the Laguerre polynomials according to (5.1.6) of \cite{S}, we can form the polynomial  $L_N^{(-1)}$ of order $N\ge1$
where,
 by (5.2.1) of \cite{S}, 
\begin{equation}\label{laguerre-1}
L_N^{(-1)}(x)=-\frac{x}{N}L_{N-1}^{(1)}(x).
\end{equation}
Continuity arguments thus show that the equivalence of (2) and (3) in
 Lemma \ref{char-zero-B1} remains valid also for $\nu=0$ and $N\ge1$ by using the
$N$ different zeros $z_1>\ldots>z_N=0$ of $L_N^{(-1)}$.

Notice that Lemma \ref{deterministic-boundary-B1} and thus the following results cannot be applied directly to the case $\nu=0$.
On the other hand, the root system $B_N$ for $\nu=0$ is closely related to the root sytem $D_N$, and thus results  for $\nu=0$
can be derived via the corresponding  results for $D_N$; see Section \ref{D-N} below.
\end{remark}

Parts (2) and (3) of Lemma \ref{char-zero-B1}  lead to the following explicit solution of the differential equation of
 Lemma \ref{deterministic-boundary-B1}:

\begin{corollary}\label{special-solution-B1}
Let $\nu>0$ and $y\in C_N^B$ the vector in Eq.~(\ref{y-max-B1}). Then for
 each $c>0$, a solution of the dynamical system in  Lemma \ref{deterministic-boundary-B1}
is given by $\phi(t,c\cdot y)= \sqrt{t+c^2}\cdot y $.
\end{corollary}

Notice that on an informal level, Corollary \ref{special-solution-B1}  holds also for $c=0$.

We now turn to the first main result of this section, a locally uniform strong LL which is analog to Theorem \ref{SLLN-A}:

\begin{theorem}\label{SLLN-B1}
 Let $\nu>0$. Let $x$ be a point in the interior of $C_N^B$, and let $y\in \mathbb R^N$. 
Let $\beta_0\ge 1/2$ with $\sqrt \beta \cdot x+y$ in the interior of $C_N^B$ for $\beta\ge \beta_0$.

For  $\beta\ge \beta_0$, consider the Bessel processes $(X_{t,k})_{t\ge0}$ of type B
  with $k=(k_1,k_2)=(\beta\cdot\nu ,\beta)$, which
start in $\sqrt \beta\cdot x+y$.
Then, for all $t>0$,
\[\sup_{0\le s\le t, \beta\ge\beta_0}\|X_{s,k}- \sqrt\beta \phi(s,x) \|<\infty \quad\quad \text{a.s..}\]
In particular,
\[X_{t,(\nu\cdot \beta,\beta)}/\sqrt \beta\to \phi(t,x) \quad\quad\text{for}\quad\quad \beta\to\infty\]
locally uniformly in $t$ a.s.~and thus locally uniformly in $t$ in probability.
\end{theorem}

\begin{proof} 
The proof is analog to that of  Theorem \ref{SLLN-A}; we only sketch the most important steps.
Recall that  $(\tilde X_{t,k}:=X_{t,k}/\sqrt \beta)_{t\ge0}$ satisfies
\[\tilde X_{t,k}^i =\frac{1}{\sqrt \beta}(y_i+B_t^i) +x_i + \int_0^t\Bigl(
\sum_{j\ne i} 
 \Bigl(\frac{1}{\tilde X_{s,k}^i-\tilde X_{s,k}^j}+\frac{1}{\tilde X_{s,k}^i+\tilde X_{s,k}^j}\Bigr) +
 \frac{\nu}{\tilde X_{s,k}^i}\Bigr) ds\]
for $i=1,\ldots,N$. 
We compare $\tilde X_{t,k}$ with the solution $Y_t=\phi(t,x)$ of
\[Y_t^i =x_i + \int_0^t\Bigl(\sum_{j\ne i}\Bigl( \frac{1}{Y_{s}^i-Y_{s}^j}+\frac{1}{Y_{s}^i+Y_{s}^j}\Bigr) +
\frac{\nu}{ Y_{s}^i}\Bigr) ds\]
for $i=1,\ldots,N$
of Lemma \ref{deterministic-boundary-B1}.

For both equations we perform  Picard iterations as in the proof of   Theorem \ref{SLLN-A}.
If we use  Lemma \ref{deterministic-boundary-B1} instead of \ref{deterministic-boundary-A}, we obtain 
that for each $t>0$, a suitable $\epsilon>0$ with  $x+y/\sqrt \beta\in U_\epsilon$ for $\beta\ge \beta_0$ , and the stopping times 
\[T_{\epsilon,k}:=\inf\{t>0:\> \tilde  X_{t,k}\not\in U_\epsilon\},\] we have
\begin{equation}\label{est-limit-B}
\sup_{s\in [0,t\wedge T_{\epsilon,k}]}\|\tilde X_{s,k}-Y_{s}\|\le \frac{1}{\sqrt{\beta}}\cdot C_t\cdot  e^{L_\epsilon t}
\end{equation}
with suitable Lipschitz constants on $U_\epsilon$ and the almost-surely finite random variable
 $C_t:=\|y\|+ \sup_{s\in [0,t]}\|B_s\| $.

We now use the modified distance 
\[ d(x,\partial C_N^B):=\max\Bigl\{\frac{(N-1) x_N}{\nu}, \> x_1-x_{2}, \ldots, \> x_{N-1}-x_{N}\Bigr\}\]
of $x\in U_\epsilon$ from $\partial C_N^B$ (which fits to the definition of $U_\epsilon$ in  Lemma \ref{deterministic-boundary-B1})
instead of the usual distance in the proof of   Theorem \ref{SLLN-A}.
Using (\ref{est-limit-B}) we then complete the proof precisely  as in the $A_{N-1}$ case.
\end{proof}

We now turn to the second freezing regime with   $k_1\to\infty$ and $k_2>0$ fixed. 
We study the normalized processes $(\tilde X_{t,k}:=X_{t,k}/\sqrt {k_1})_{t\ge0}$ with
\begin{equation}\label{SDE-B2-normalized}
d\tilde X_{t,k}^i =\frac{1}{\sqrt{k_1}}dB_t^i + \frac{k_2}{k_1}\sum_{j\ne i} 
\Bigl( \frac{1}{\tilde X_{t,k}^i-\tilde X_{t,k}^j} +   \frac{1}{\tilde X_{t,k}^i+\tilde X_{t,k}^j}\Bigr)dt
+ \frac{1}{\tilde X_{t,k}^i} dt \end{equation}
for $i=1,\ldots,N$.
We again compare   $\tilde X_{t,k}$ with the solutions of a deterministic dynamical system
which is much easier than in the previous cases.

\begin{lemma}\label{deterministic-boundary-B2}
Let $k_2>0$. For $\epsilon>0$ consider the open sets
 $U_\epsilon:=\{x\in C_N^B:\> d(x,\partial C_N^B)>\epsilon\}$. Then the function 
\[H:U_\epsilon\to \mathbb R^N, \quad x\mapsto (1/x_1,\ldots, 1/x_N)\]
is Lipschitz continuous on $U_\epsilon$ with  constant $\epsilon^{-2}$. Moreover,
 for each starting point $x_0\in U_\epsilon$, the solution $\phi(t,x_0)$
of the dynamical system $\frac{dx}{dt}(t) =H(x(t))$ is given by
 \[\phi(t,x_0)=\Bigl(\sqrt{2t+x_{0,1}^2}, \ldots,\sqrt{2t+x_{0,N}^2}\Bigr) \] with  $\phi(t,x_0)$ in the
interior of  $C_N^B$ for $t\ge0$.
\end{lemma}

The proof of this lemma is straightforward. Notice that on an informal level, the dynamical system of Lemma \ref{deterministic-boundary-B2}
has the solution $\phi(t,x_0)$  for all starting points $x_0\in C_N^B$. We now turn to the strong limiting law.

\begin{theorem}\label{SLLN-B2}
 Let $k_2>0$. Let $x$ be a point in the interior of $C_N^B$, and let $y\in \mathbb R^N$. 
Let $k_0\ge 1/2$ large enough such that $\sqrt{k_1} \cdot x+y$ is in the interior of $C_N^B$ for $k_1\ge k_0$.

For  $k_1\ge k_0$, consider the Bessel processes $(X_{t,k})_{t\ge0}$ of type $B_N$  with $k=(k_1,k_2)$, which
start in $\sqrt{k_1}\cdot x+y$.
Then, for all $t>0$,
\[\sup_{0\le s\le t, k_1\ge k_0}\|X_{t,k}- \sqrt{k_1} \phi(t,x) \|<\infty  \quad\quad a.s..\]
In particular,
\[X_{t,(k_1,k_2)}/\sqrt{k_1} \to \phi(t,x) \quad\quad\text{for}\quad\quad k_1\to\infty\]
locally uniformly in $t$ a.s. and thus locally uniformly in $t$ in probability.
\end{theorem}

\begin{proof} 
The proof is analog to that of  Theorems \ref{SLLN-A} and \ref{SLLN-B1}; we only sketch the main steps and 
describe the  differences.
Notice that  $(\tilde X_{t,k}:=X_{t,k}/\sqrt{k_1})_{t\ge0}$ satisfies
\[\tilde X_{t,k}^i =\frac{1}{\sqrt{k_1}}(y_i+B_t^i) +\frac{k_2}{k_1} \int_0^t
\sum_{j\ne i} 
 \Bigl(\frac{1}{\tilde X_{s,k}^i-\tilde X_{s,k}^j}+\frac{1}{\tilde X_{s,k}^i+\tilde X_{s,k}^j}
\Bigr)ds 
+x_i+  \int_0^t\frac{1}{\tilde X_{s,k}^i}ds\]
for $i=1,\ldots,N$. 
We compare $\tilde X_{t,k}$ with the solution $Y_t=\phi(t,x)$ of
\[Y_t^i =x_i + \int_0^t\frac{1}{ Y_{s}^i}ds \quad\quad(i=1,\ldots,N)\]
of Lemma \ref{deterministic-boundary-B2}.

For both equations we perform  Picard iterations as in the proof of   Theorem \ref{SLLN-A}.
We notice that for any given time $t\ge0$ and given $k_0$ as above, we can find a small $\epsilon>0$ such that
the deterministic solution $\phi(s, x+y/\sqrt{k_1})$ of Lemma  \ref{deterministic-boundary-B2} 
is contained in $U_\epsilon$ for all $k_1\ge k_0$ and all $s\in[0,t]$.
If we consider the stopping times 
\[T_{\epsilon,k}:=\inf\{t>0:\> \tilde  X_{t,k}\not\in U_\epsilon\}\] we obtain as in the proof of Theorem \ref{SLLN-A} that
\begin{equation}\label{est-limit-B2}
\sup_{s\in [0,t\wedge T_{\epsilon,k}]}\|\tilde X_{s,k}-Y_{s}\|\le \frac{1}{\sqrt{k_1}}\cdot C_t\cdot  e^{t/\epsilon^2}
\end{equation}
with  the a.s. finite random variable
 $C_t:=\|y\|+ \sup_{s\in [0,t]}\|B_s\| $.

We complete this proof by following the steps in the proof of  Theorem \ref{SLLN-A}.
\end{proof}

\section{A central limit theorem for the root system B}

In this section we show that the locally uniform limit law in Theorem~\ref{SLLN-B2} above can be used to
derive a central limit theorem. This result generalizes the case $B_1$ for classical one-dimensional 
Bessel processes where this is a classical and well-known result; see Remark \ref{one-dim-case}   below.

\begin{theorem}\label{CLT-B2}
 Let $k_2>0$. Let $x$ be a point in the interior of $C_N^B$, and let $y\in \mathbb R^N$. 
Let $k_0\ge 1/2$ large enough that $\sqrt{k_1} \cdot x+y$ is in the interior of $C_N^B$ for $k_1\ge k_0$.
For  $k_1\ge k_0$, consider the Bessel processes $(X_{t,k})_{t\ge0}$ of type B  with $k=(k_1,k_2)$, which
start in $\sqrt{k_1}\cdot x+y$.
Then, for all $t>0$, 
\[X_{t,(k_1,k_2)}-\sqrt{k_1} \cdot \Bigl(\sqrt{2t+x_{1}^2}, \ldots,\sqrt{2t+x_{N}^2}\Bigr)\]
tends in distribution for  $k_1\to\infty$ to the normal distribution 
\[N\left(0,\textnormal{diag}\left(\frac{t^2+tx_1^2}{2t+x_1^2}, \ldots,\frac{t^2+tx_N^2}{2t+x_N^2} \right)\right).\]
\end{theorem}

\begin{proof} 
Consider the process $\bigl(Z_{t,k}:= ((X^1_{t,k})^2, \ldots ,(X^N_{t,k})^2)\bigr)_{t\ge0}$. The It{\^o} formula and the SDE
for $(X^i_{t,k})_{t\ge0}$ show for $i=1,\ldots, N$ that
\begin{align}
dZ_{t,k}^i&= 2X^i_{t,k}\> dX^i_{t,k} \> +\> dt \notag\\
&= 2X^i_{t,k}\>dB^i_{t} \> +\> 2k_2\sum_{j\ne i} 
\Bigl( \frac{X^i_{t,k}}{X^i_{t,k}-X^j_{t,k}} +\frac{X^i_{t,k}}{X^i_{t,k}+X^j_{t,k}} \Bigr) dt \> +\>
(2k_1+1)dt.\notag
\end{align}
Hence,
\begin{align}\label{clt-main-rel}
 \frac{1}{\sqrt{k_1}}&\Bigl(Z_{t,k}^i-Z_{0,k}^i-(2k_1+1)t\Bigr)=\notag\\
&=2 \int_0^t \frac{X^i_{s,k}}{\sqrt{k_1}} \> dB^i_{s}\> +\>
\frac{2k_2}{\sqrt{k_1}}\int_0^t \sum_{j\ne i} 
\Bigl( \frac{X^i_{s,k}}{X^i_{s,k}-X^j_{s,k}} +\frac{X^i_{s,k}}{X^i_{s,k}+X^j_{s,k}} \Bigr) ds.
\notag\\
&=2 \int_0^t \frac{X^i_{s,k}}{\sqrt{k_1}} \> dB^i_{s}\> +\>
\frac{4k_2}{\sqrt{k_1}}\int_0^t \sum_{j\ne i} 
 \frac{(X^i_{s,k})^2}{(X^i_{s,k})^2-(X^j_{s,k})^2} ds.
\end{align}
As $X^i_{s,k}/\sqrt{k_1}\to\phi(s,x)$ locally uniformly in
 probability by Theorem \ref{SLLN-B2} with the function $\phi$ from Lemma \ref{deterministic-boundary-B2}, 
we obtain from standard results on stochastic integrals (see e.g. Section II.4 of \cite{P}) that
\[ \int_0^t \frac{X^i_{s,k}}{\sqrt{k_1}} \> dB^i_{s}\longrightarrow \int_0^t  \phi(s,x)\> dB^i_{s}\]
locally uniformly in $t$ in probability. Moreover, by the same argument,
 the integrand of the second integral of the r.h.s.~of
(\ref{clt-main-rel}) converges also to a finite, continuous deterministic function, that is, the second summand
 of the r.h.s.~of
(\ref{clt-main-rel}) converges to $0$ locally uniformly in $t$ in probability. Hence, using the initial condition,
we see that
\[\frac{1}{\sqrt{k_1}}(Z_{t,k}^i-(x_i\sqrt{k_1}+y_i)^2 - (2k_1+1)t) \longrightarrow 2\cdot \int_0^t 
\sqrt{2s+x_i^2}\>  dB^i_{s}\]
 in probability for $k\to\infty$ and $i=1,\ldots,N$. As the limits are $N(0,4t^2+4tx_i^2)$-distributed and independent
for $i=1,\ldots,N$ we conclude that
\begin{equation}\label{limit-normal-2}
\frac{1}{\sqrt{k_1}}\Bigl(Z_{t,k}^1-x_1^2k_1-2\sqrt{k_1}x_1y_1 -2k_1t, \ldots,
Z_{t,k}^N-x_N^2k_1-2\sqrt{k_1}x_Ny_N -2k_1t\Bigr)
\end{equation}
tends in distribution to the $N$-dimensional normal distribution 
\begin{equation}\label{limit-normal-1}
N(0,\textnormal{diag}(4 t^2+4tx_1^2, \ldots,4t^2+4tx_N^2)) .
\end{equation}

In order to obtain a CLT for the original variables $X^i_{t,k}$, we use the definition of $Z_{t,k}$ and observe that
\begin{align}
\frac{1}{\sqrt{k_1}}&\Bigl((X_{t,k}^i)^2- k_1(x_i^2+2t +2x_iy_i/\sqrt{k_1})\Bigr)\notag\\
&= \Bigl(X_{t,k}^i- \sqrt{k_1}\cdot\sqrt{x_i^2+2t +2x_iy_i/\sqrt{k_1}}\Bigr) \notag\\
&\quad\times \frac{1}{\sqrt{k_1}}
 \Bigl(X_{t,k}^i+ \sqrt{k_1}\cdot\sqrt{x_i^2+2t +2x_iy_i/\sqrt{k_1}}\Bigr)\notag
\end{align}
where the second factor in the r.h.s.~tends in probability to $2\sqrt{x_i^2+2t}$ for  $k\to\infty$ and $i=1,\ldots,N$
by  Theorem \ref{SLLN-B2}. This, the CLT for $Z_{t,k}$ in the first part of the proof, and Slutsky's lemma applied to the quotient
\[\frac{(Z_{t,k}^1-x_1^2k_1-2\sqrt{k_1}x_1y_1 -2k_1t)/\sqrt{k_1}}{\Bigl(X_{t,k}^i+ \sqrt{k_1}\cdot\sqrt{x_i^2+2t +2x_iy_i/\sqrt{k_1}}\Bigr)/\sqrt{k_1}}=X_{t,k}^i- \sqrt{k_1}\cdot\sqrt{x_i^2+2t +2x_iy_i/\sqrt{k_1}}\]
yield that
\[\Bigl(X_{t,k}^1-\sqrt{k_1}\sqrt{x_1^2+2t}, \ldots, X_{t,k}^N-\sqrt{k_1}\sqrt{x_N^2+2t}\Bigr)\]
converges in distribution to the  normal distribution given in the statement.
\end{proof}

\begin{figure}[!t]
\[
\begin{array}{cc}
\includegraphics[width=0.45\textwidth]{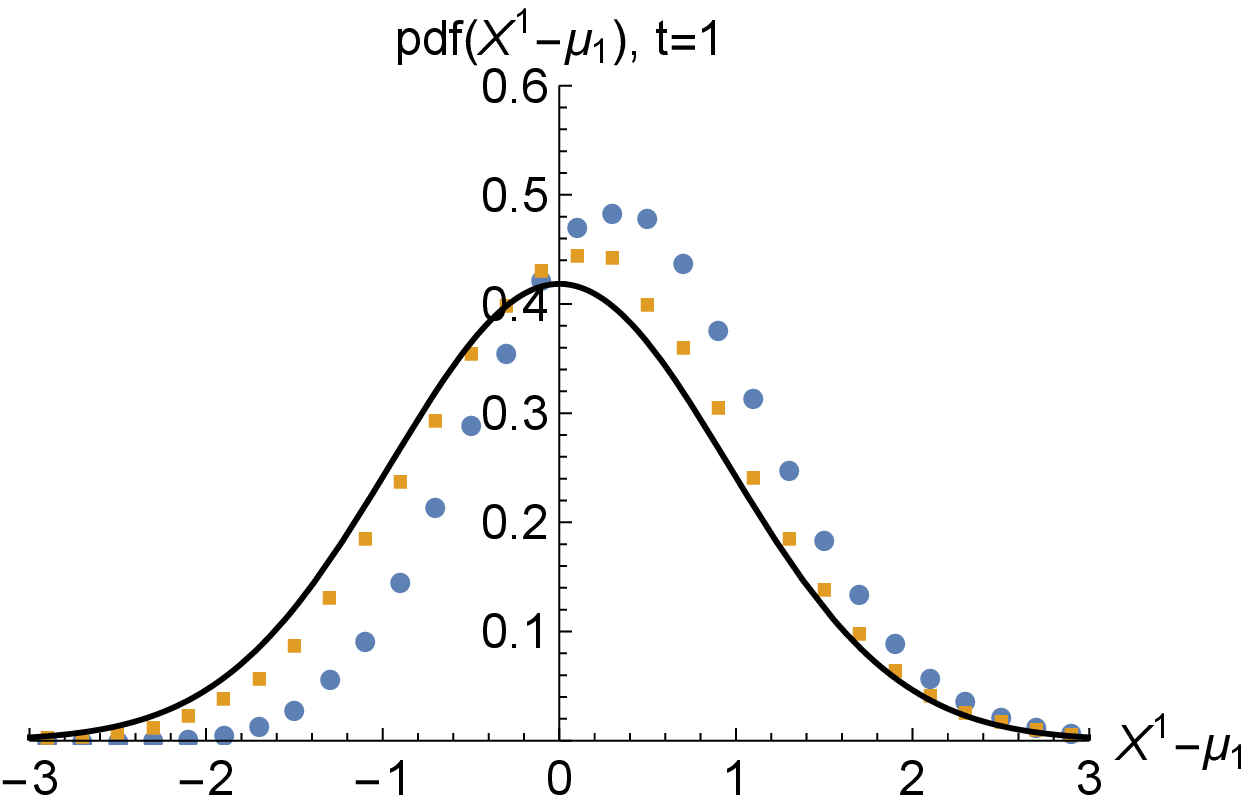}&\includegraphics[width=0.45\textwidth]{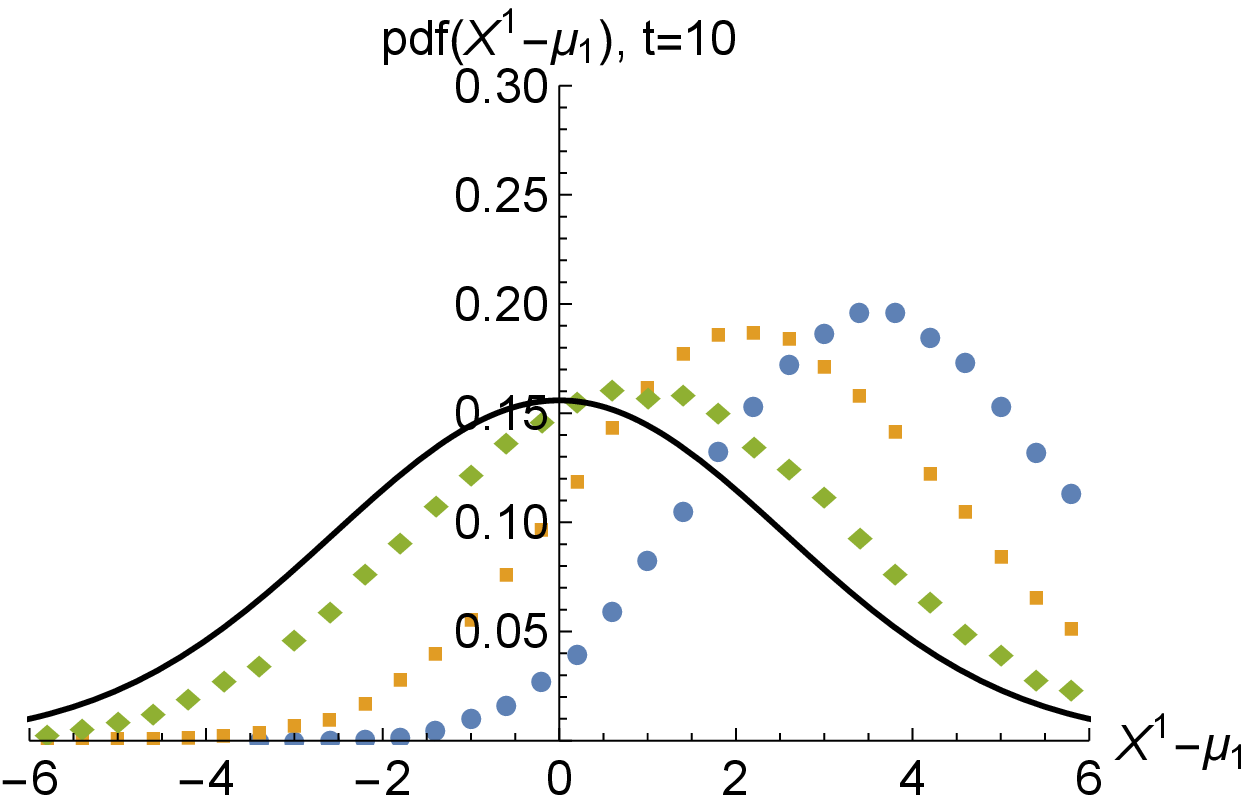}\\
\text{a)}&\text{b)}\\
\includegraphics[width=0.45\textwidth]{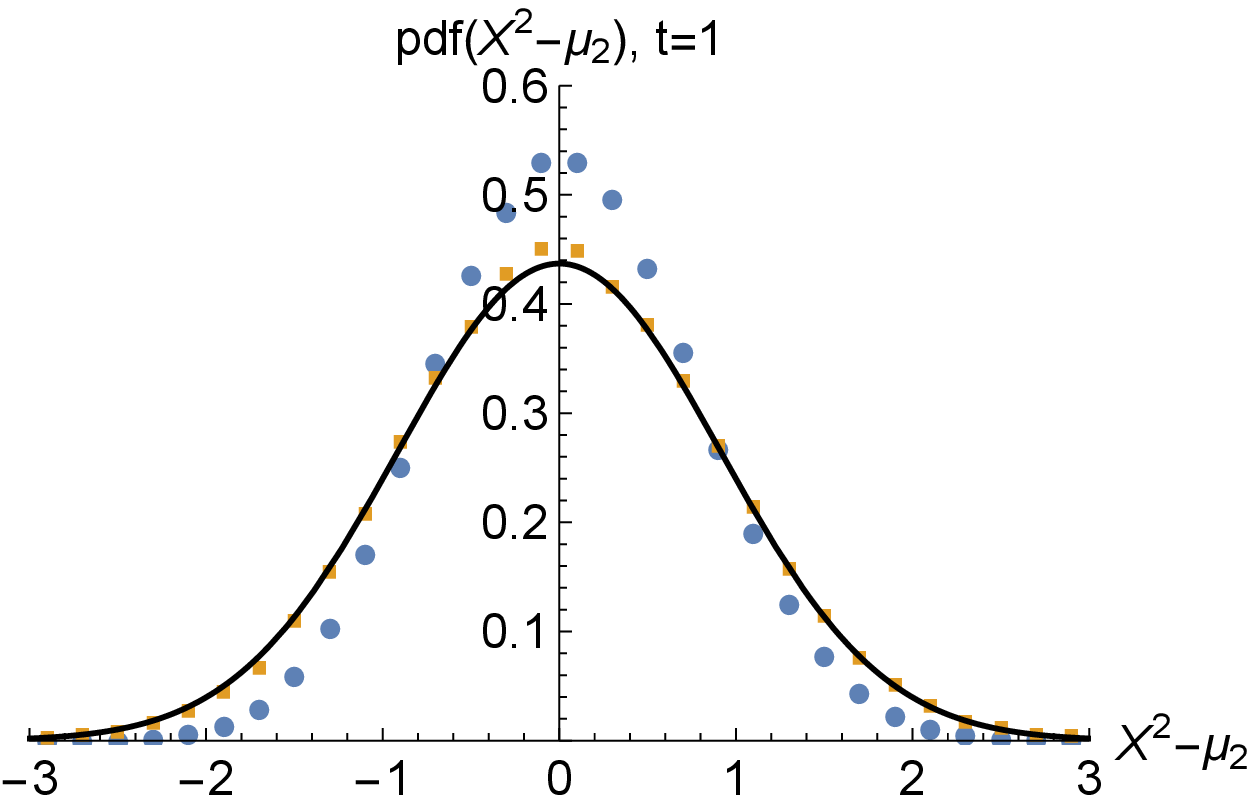}&\includegraphics[width=0.45\textwidth]{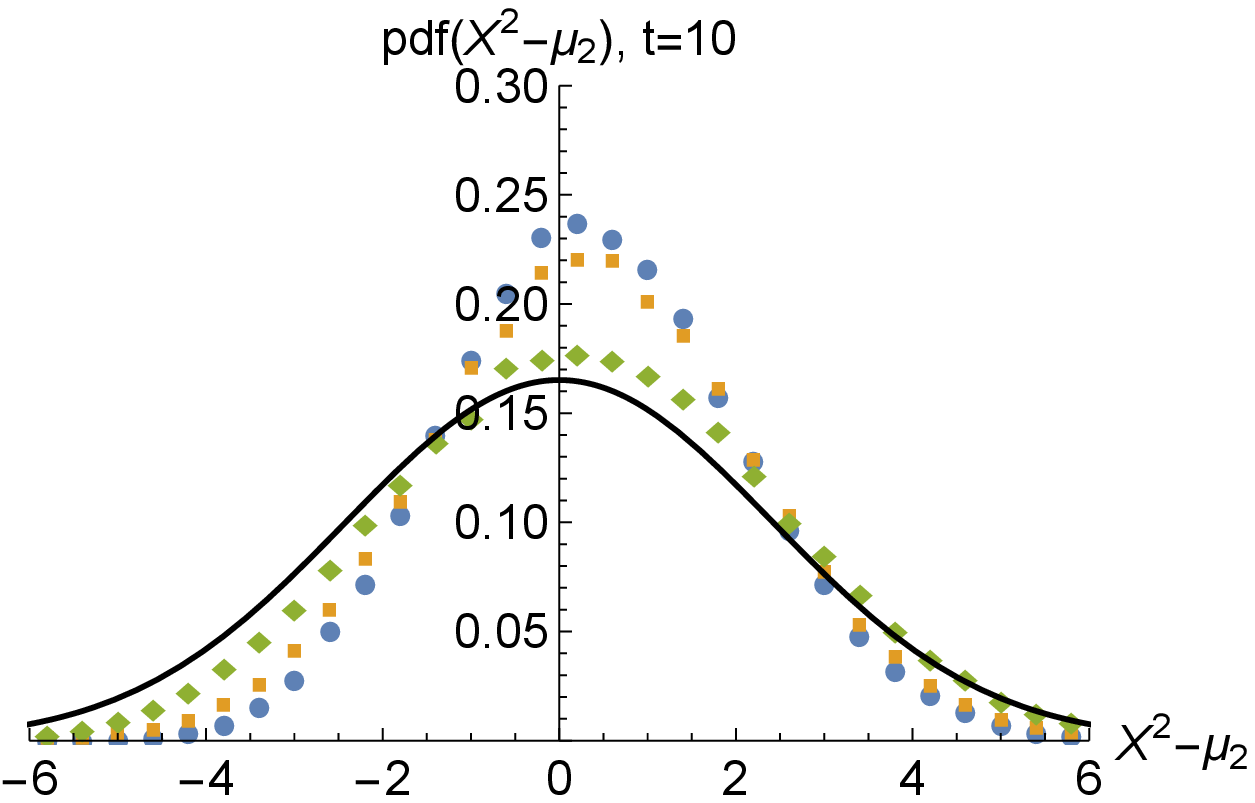}\\
\text{c)}&\text{d)}\\
\includegraphics[width=0.45\textwidth]{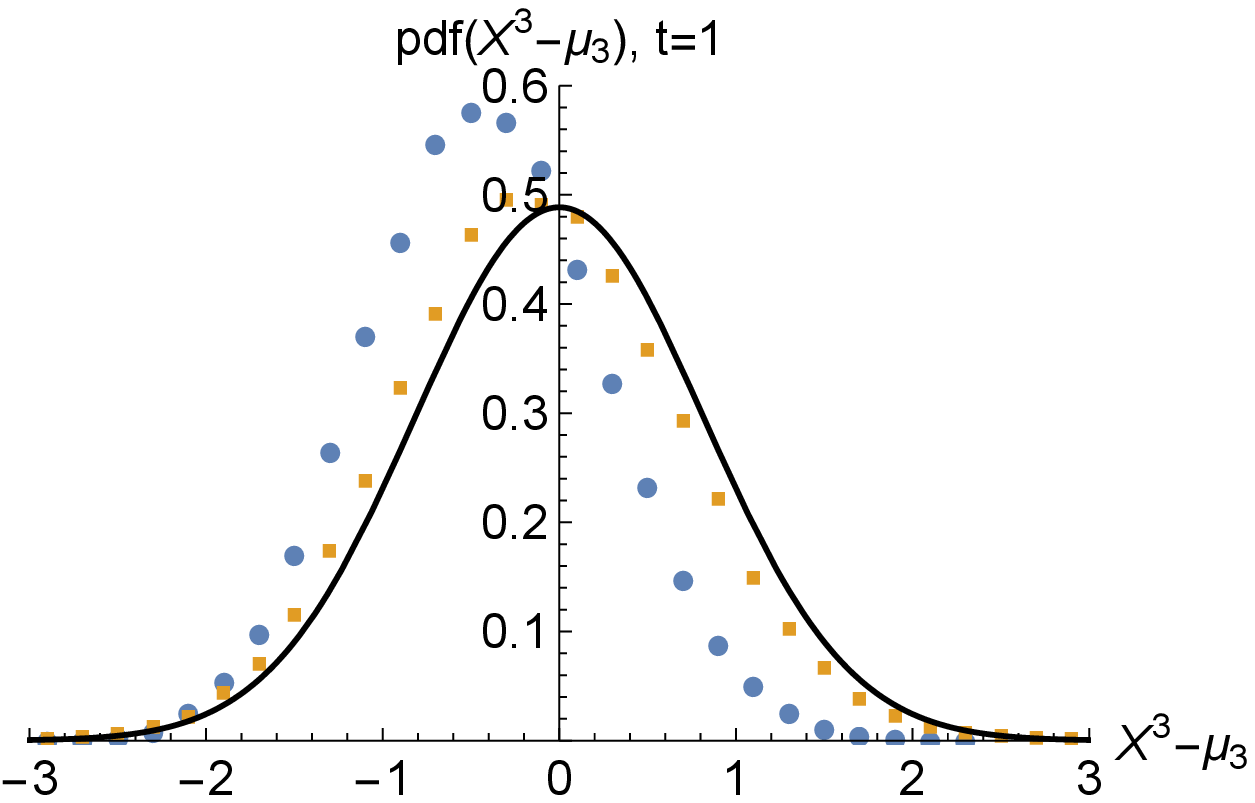}&\includegraphics[width=0.45\textwidth]{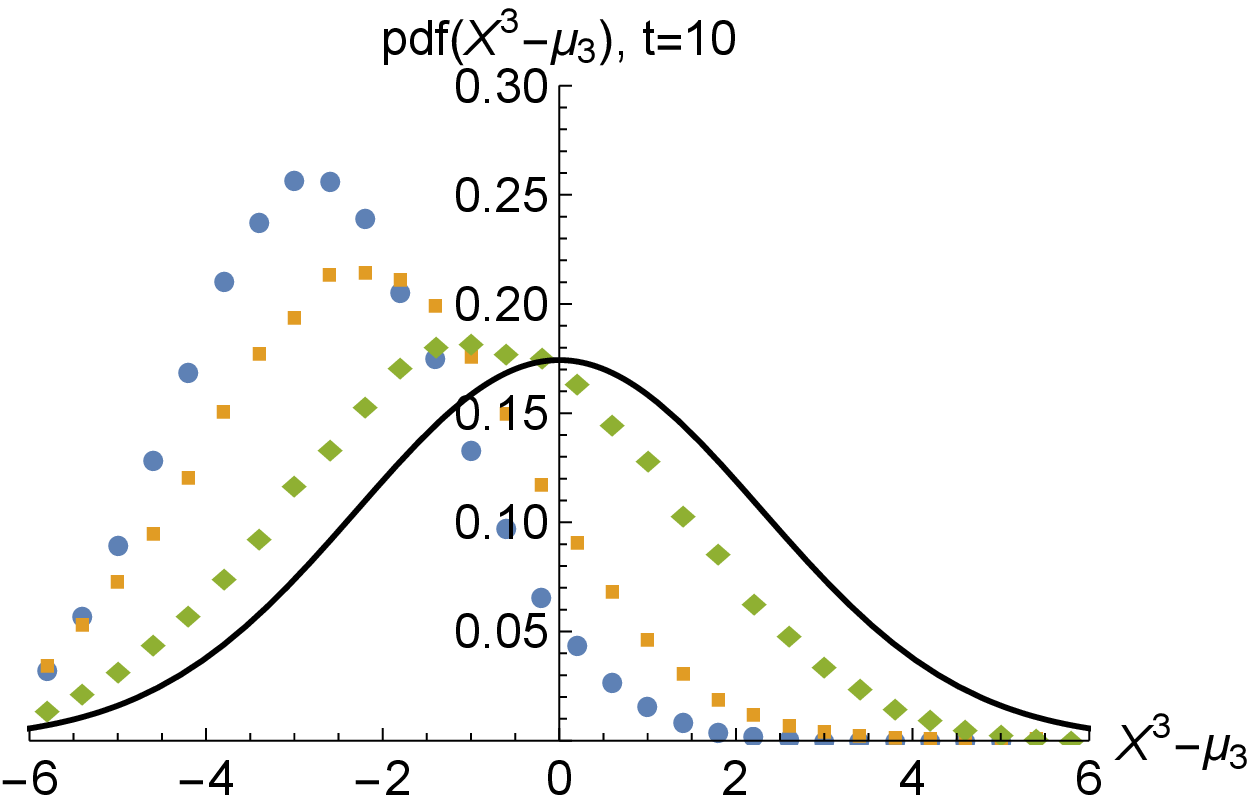}\\
\text{e)}&\text{f)}
\end{array}
\]
\caption{\label{fig-CLT-B2}} Illustration of Theorem~\ref{CLT-B2} for $N=3$ and $k_2=1$ with $(x_1,x_2,x_3)=(3,2,1)$.
The $i$-th row corresponds to $X_{t,k}^i$, while the first column corresponds to $t=1$ and the second column corresponds to $t=10$.
In each plot, the solid lines correspond to the limiting normal distribution, while the blue circles, yellow squares and green rhombi correspond to the distribution of $X_{t,k}^i-\sqrt{k_1(2t-x_i^2)}$ obtained in numerical simulations for $k_1=5$, 50, and 500, respectively.
\end{figure}

We illustrate this theorem for the case $N=3$ in Figure~\ref{fig-CLT-B2}, where the limiting distribution is compared with numerical simulations of the process for several values of $k_1$ and $t$. As expected from the theorem, the (centered) distribution of each particle approaches the limiting normal distribution as $k_1$ grows, but there is a clear bias in the numerical results. By observing the plots corresponding to $t=10$ (plots b), d), and f) in the figure), it is apparent that a larger value of $k_1$ is necessary to approach the limiting distribution at larger times. This means that the bias is not an effect of the starting position of the process, but rather an accumulating effect of the second term in the last line of \eqref{clt-main-rel}, which represents the repulsion between particles. Indeed, the rightmost particle ($X_{t,k}^1$) is pushed to the right and the leftmost particle ($X_{t,k}^3$) is pushed to the left. In the case of $X_{t,k}^2$, the bias is much smaller because the repulsion from $X_{t,k}^1$ and $X_{t,k}^3$ cancel each other partially, leading to a much faster convergence to the limiting distribution.

\begin{remark}\label{one-dim-case}
 We briefly discuss the CLT \ref{CLT-B2} for $N=1$ which is the case where \ref{CLT-B2}  is reduced to a known classical CLT for
classical one-dimensional Bessel processes. To explain this fix  $\tilde x\in]0,\infty[$.
 Consider independent one-dimensional Brownian motions $(B_t^l)_{t\ge0}$ starting in $0$ for $l\in\mathbb N$.
It is well-known (see e.g. Sections VI.3 and XI.1 of \cite{RY}) 
that for $d \in\mathbb N$ the sums of shifted squares 
\[\Bigl( S_{t,d}:=\sum_{l=1}^d (B_t^l+\tilde x)^2\Bigr)_{t\ge0}\]
are squares of classical one-dimensional Bessel processes, i.e., of Bessel processes of type $B_1$ with multiplicity
$k_1=(d-1)/2$ ($k_2$ is irrelevant here).

Now fix $t>0$. Then $S_{t,d}$ is a sum of $d$ iid random variables with mean $t+\tilde x^2$ and variance $2t^2+4t\tilde x^2$.
Therefore, by the classical CLT for sums of iid random variables,
\begin{equation}\label{CLT-1-dim}
\frac{S_{t,d}-d(t+\tilde x^2)}{\sqrt{d}}\to N(0,2t^2+4t\tilde x^2)
\end{equation}
for $d\to \infty$ in distribution.

This CLT corresponds perfectly with the convergence of (\ref{limit-normal-2}) to the distribution 
 (\ref{limit-normal-1}) if one takes into account that we have $k_1=(d-1)/2$ which implies that the  point $x$
 in Theorem \ref{CLT-B2}  is related to $\tilde x$ by $d\tilde x^2=k x^2$, i.e., $\tilde x^2=\frac{d-1}{2d}x^2$.
We notice that this approach to the one-dimensional CLT (\ref{CLT-1-dim}) also works for
 any real parameter $d\in[1,\infty[$ and also for the starting point $\tilde x=0$.
\end{remark}

We also notice that for 
the cases $k_2=1/2,1,2$, the central limit theorem is related to a central limit theorem for 
Wishart distributions on the cones $\Pi_N(\mathbb F)$ of all $N\times N$-dimensional, positive semidefinite 
matrices over the fields $\mathbb F=\mathbb R, \mathbb C$ and the skew-field of quaternions $\mathbb H$ respectively.
We discuss this extension of the preceding remark briefly.      

\begin{remark}
Fix one of the (skew-)fields $\mathbb F=\mathbb R, \mathbb C, \mathbb H$ as before with the real dimension $d=1,2,4$ respectively. For
integers $p\in\mathbb N$ consider the vector space $M_{p,N}(\mathbb F)$ of all $p\times N$-matrices over $\mathbb F$ with
 the real dimension $dpN$.
Choose the standard basis there with $d$ basis vectors in each entry, and consider the $dpN$-dimensional associated Brownian motion
 $(B_t^p)_{t\ge0}$ on  $M_{p,N}(\mathbb F)$. If we write $A^*:=\overline A^T\in M_{N,p}(\mathbb F)$ for matrices $A\in M_{p,N}(\mathbb F)$
 with the usual conjugation on $\mathbb F$, the process $(Z_t^p:= (B_t^p)^*B_t^p)_{t\ge0}$ becomes 
a Wishart process on the cone $\Pi_N(\mathbb F)$
of all $N\times N$ positive semidefinite matrices over  $\mathbb F$  with shape parameter $p$; see \cite{Bru,DDMY} for details on
 Wishart processes.

Let $\sigma_N:\Pi_N(\mathbb F)\to C_N^B$ be the mapping which relates to each matrix in $\Pi_N(\mathbb F)$ its  ordered spectrum.
Then, it is well-known that $(\sqrt{\sigma_N(Z_t^p)})_{t\ge0}$ is a Bessel process on  $C_N^B$ of type $B_N$ with multiplicity
$(k_1,k_2):=((p-N+1)\cdot d/2, d/2)$ where the symbol $\sqrt{.}$ means taking square roots in each component;
 see e.g. \cite{BF,R3} for details.

We thus conclude that the CLT  \ref{CLT-B2} for $k_2=1/2, 1,2$ corresponds to a CLT for Wishart distributions on  $\Pi_N(\mathbb F)$
with fixed time parameters where the shape parameters $p$ tend to $\infty$. Notice that the distributions
 $\mu_t^p:=P_{Z_t^p}\in M^1(\Pi_N(\mathbb F))$ of $Z_t^p$ satisfy $\mu_t^{p_1}*\mu_t^{p_2}=\mu_t^{p_1+p_2}$ for $p_1,p_2\in\mathbb N$ with
 the usual convolution of measures on the vector space of all $N\times N$ Hermitian matrices over $\mathbb F$ by the very construction of the 
random variables $Z_t^p$. Moreover, this convolution relation holds for all real parameters $p$ which are sufficiently large. 
We thus may apply the classical law of large numbers and CLT for sums of iid random variables on finitely dimensional vector spaces to obtain
 LLs and a CLT for Wishart distributions for $p\to\infty$. If one computes the mean vectors and covariance matrices for  $Z_t^p$ 
one obtains readily that this classical CLT  for $p\to\infty$ on the level of  Hermitian matrices corresponds to Theorem  \ref{CLT-B2}
on the Weyl chamber $C_N^B$.

We also remark that in this setting  there are related LLs and CLTs for radial random walks $(X_n^p)_{n\ge0}$ 
on  the vector space $M_{p,N}(\mathbb F)$ when the dimension parameter $p$ as well as the time parameter $n$ tend to $\infty$ in a 
coupled way; see \cite{G,RV3,V1}. We also mention that the CLT \ref{CLT-B2}
 has some relations with limit theorems of Bougerol \cite{B} for noncompact Grassmann manifolds over $\mathbb F$
 when the dimensions tend to infinity.
\end{remark}

The strong LLs \ref{SLLN-A}, \ref{SLLN-B1}, and \ref{SLLN-D} also admit central limit theorems similar to Theorem
 \ref{CLT-B2}. These results, whose proofs are also based on these strong LLs, 
 are more complicated and will be presented in \cite{VW}.
To get a brief impression,
 we fix a root system, a 
multiplicity $k$ (which might be 2-dimensional in the case $B_N$), and the corresponding Bessel processes $(X_{t,k})_{t\ge0}$.
For each function $F\in C^{(2)}(\mathbb R^N)$ we obtain from the It{\^o} formula and the general SDE (\ref{SDE-general})
that
\begin{align}
dF(X_{t,k})&= \nabla F(X_{t,k}) \> dX_{t,k} +\frac{1}{2}\Delta  F(X_{t,k}) \> dt\notag\\
&= \nabla F(X_{t,k}) \> dB_t +\frac{1}{2}\Bigl( ( \nabla F\cdot \nabla (\ln w_k))(X_{t,k})  +\Delta  F(X_{t,k})\Bigr) dt
\notag
\end{align}
where $\nabla (\ln w_k))$ has the form $ k\cdot H(x)$ for the root systems of type $A_{N-1}$, $D_N$, and the form
$ k_1\cdot H_1(x)+ k_2\cdot H_2(x)$ for the root system $B_N$ with suitable functions $H,H_1,H_2$. We now search for
$F\in C^{(2)}(\mathbb R^N)$, for which
\[ (\nabla F\cdot \nabla (\ln w_k))(x)\]
is independent of $x\in\mathbb R^N $. Similar to the proof of  Theorem \ref{CLT-B2} we then obtain a CLT for 
$F(X_{t,k})$ for starting points in the interior when the multiplicity or a part of it 
tends to infinity.

It was noticed by J.~Woerner that in all cases, a non-trivial example of a function $F$ with the desired properties is given by
$F(x):=\|x\|_2^2$. This can be checked easily for all root systems. This observation leads readily to the following CLT:

\begin{proposition}\label{clt-norm}
Consider  Bessel processes  $(X_{t,k})_{t\ge0}$ of types $A_{N-1}$, $B_N$, or $D_N$ as above in 
the strong LLs \ref{SLLN-A}, \ref{SLLN-B1}, and \ref{SLLN-D} with the starting points given therein. 
Let $\gamma>0$ be such that the weight function $w_k$ 
is homogeneous of degree  $2\gamma$ 
(see (\ref{def-gamma}) for the cases $A_{N-1}$, $B_N$ and the beginning of the next section for the case $D_N$ for precise formulas).

Then, for each $t>0$ and for all multiplicities $k$ with $\gamma\to\infty$,
\[\|X_{t,k}\|_2-\sqrt{\gamma+(N-1)/2}\cdot\sqrt{2t+\|x\|_2^2}\] 
converges in distribution to
\[N\Bigl(0, \frac{t^2+t\|x\|_2^2}{2t+\|x\|_2^2}\Bigr)\]
\end{proposition}

\begin{proof}
The proof can be carried out by using the function $F(x):=\|x\|_2^2$ as explained above.

We give a second proof. It is well-known by \cite{RV1} that $(\|X_{t,k}\|_2)_{t\ge0}$ is a classical one-dimensional 
Bessel process of
 type $B_1$ with multiplicity $\gamma+(N-1)/2$. If we apply  Theorem \ref{CLT-B2} to this case, the statement follows.
\end{proof}

Besides of the CLT \ref{clt-norm} there exist other CLTs. For instance, for the case $A_{N-1}$, Eq.~(\ref{SDE-A})
implies that the center of gravity is
\[\frac{1}{N}\sum_{i=1}^N dX_{t,k}^i =\frac{1}{N}\sum_{i=1}^N dB_t^i,\]
i.e., it is a Brownian motion up to scaling. Also for the case $A_1$ with $2$ particles, a  CLT can be derived in a simple way.

\section{Strong limiting law for the root system $D_N$}\label{D-N}

We briefly study a limit theorem for Bessel processes of type $D_N$ next. We recapitulate that the root system
is given here by 
\[D_N=\{\pm e_1\pm e_j: \quad 1\le i<j\le N\}\]
with associated closed Weyl chamber 
\[C_N:=C_N^D=\{x\in\mathbb R^N: \quad x_1\ge \ldots\ge x_{N-1}\ge |x_N|\}.\]
$C_N^D$ may be seen as a doubling of $C_N^B$ w.r.t.~the last coordinate. We have a one-dimensional multiplicity $k\ge0$.
The weight function from (\ref{density-general}) is given by
\[w_k(x):=w_k^D(x):= \prod_{i<j}(x_i^2-x_j^2)^{2k},\]
the associated constant $\gamma$ by $\gamma_D:= kN(N-1)$, 
and the generator of the transition semigroup by
\begin{equation}\label{def-L-D} Lf:= \frac{1}{2} \Delta f +
 k \sum_{i=1}^N \sum_{j\ne i} \Bigl( \frac{1}{x_i-x_j}+\frac{1}{x_i+x_j}  \Bigr)
 \frac{\partial}{\partial x_i}f, \end{equation}
c.f. (\ref{def-L-A}) and  (\ref{def-L-B}) for the cases $A_{N-1}$ and $B_N$. For further details on the $D_N$ case we refer to
\cite{Dem1}.

Consider the Bessel process $(X_{t,k})_{t\ge0}$ of type $D_N$ which 
starts at time $0$ from the origin, $0\in C_N^D$. In this case, 
 the SDE (\ref{SDE-general})
 reads as  (\ref{SDE-B}) with $k_1:=0$, $k_2:=k$.
Moreover,
 by  (\ref{density-general}),
the random variable $X_{t,k}/\sqrt{kt}$ ($t>0$) has the  Lebesgue density
\begin{equation}\label{density-D}
\textrm{const}(k)\cdot\exp\Bigl(k\Bigl(-\|y\|_2^2 + 2\sum_{i<j}\ln(y_1^2-y_j^2)\Bigr)\Bigr)=:\textrm{const}(k)\cdot \exp(k\cdot W_D(y))
\end{equation}
on $C_N^D$. Similar to the results above in the cases $A_{N-1}$ and $B_N$, we have:

\begin{lemma}\label{char-zero-D}
 For $y\in C_N^D$, the following statements  are equivalent:
\begin{enumerate}
\item[\rm{(1)}] The function 
$W_D(x):=2\sum_{ i<j} \ln(x_i^2-x_j^2) -\|x\|^2/2$
 is maximal at $y\in C_N^B$;
\item[\rm{(2)}] $y_N=0$, and
for $i=1,\ldots,N-1$, 
\[4 \sum_{j: j\ne i} \frac{1}{y_i^2-y_j^2} =1;\] 
\item[\rm{(3)}] If $z_1^{(1)}>\ldots>z_{N-1}^{(1)}>0$ are the $N-1$ ordered zeros of 
 the classical  Laguerre polynomial $L_{N-1}^{(1)}$, then 
\begin{equation}\label{y-max-D}
2(z_1^{(1)},\ldots,z_{N-1}^{(1)},0)=(y_1^2,\ldots,y_N^2).
\end{equation}
\end{enumerate}
\end{lemma}

\begin{proof}
Clearly $W_D(y)$ tends to $-\infty$ for $y\in C_N^D$ with $\|y\|\to\infty$ and for the case where $y$ tends to some point
in $\partial  C_N^D$. This shows that $W_D$ admits a global maximum on $C_N^D$ which is in the interior of  
$C_N^D$. Each candidate for a maximum
satisfies
\begin{equation}\label{cond-D}
-y_i+4 \sum_{j: j\ne i} \frac{y_i}{y_i^2-y_j^2} =0 \quad\quad (i=1,\ldots,N).
\end{equation}
Using $y\in C_N^D$ we see easily that $y_N=0$ (as otherwise the l.h.s. of (\ref{cond-D})
is negative for $i=N$). Moreover, as $y\not\in \partial  C_N^D$, we have $y_i>0$ for $i=1,\ldots,N-1$, that is, we 
obtain the condition in (2). If we have  the equivalence of (2) and (3), 
we see that we only have one candidate for a maximum.
This shows that (1) and (2) are equivalent. Finally, the equivalence of (2) and (3) is shown in Remark \ref{remark-b-nu0}.
\end{proof}

Lemma \ref{char-zero-D} and the explicit densities (\ref{density-D}) of $X_{t,k}/\sqrt{kt}$ immediately imply the following
weak limiting law for $X_{t,k}$ for $k\to\infty$ for start in $0$ which is analog to the LLs in 
\cite{AKM1,AKM2,AM}:

\begin{corollary}\label{lln-D-start-0}
 Consider the Bessel processes $(X_{t,k})_{t\ge0}$ of type $D_N$ which start at time $0$ in the origin $0\in C_N^D$.
Then, for each $t>0$,
$X_{t,k}/\sqrt{kt}\to y$ in probability for $k\to\infty$, where $y\in C_N^D$ is the vector in Lemma \ref{char-zero-D}.
\end{corollary}

The LLs for a starting point in the interior of the Weyl chamber as in Theorems~\ref{SLLN-A}, \ref{SLLN-B1}, and \ref{SLLN-B2} can be also derived
for the root system $D_N$. For this, we again  compare   $\tilde X_{t,k}:=  X_{t,k}/\sqrt k$ with solutions of a deterministic dynamical system.

\begin{lemma}\label{deterministic-boundary-D}
 For $\epsilon>0$ consider the open subsets
 $U_\epsilon:=\{x\in C_N^D:\> d(x,\partial C_N^D)>\epsilon\}$.
Then
 the function 
\[H:U_\epsilon\to \mathbb R^N, \quad x\mapsto 
\left(\begin{matrix}\sum_{j\ne1}\Bigl( \frac{1}{x_1-x_j}+   \frac{1}{x_1+x_j}\Bigr)\\
\vdots\\
\sum_{j\ne N} \Bigl(\frac{1}{x_N-x_j}+   \frac{1}{x_N+x_j}\Bigr)
\end{matrix}\right)\]
is Lipschitz continuous on $U_\epsilon$ with Lipschitz constant $L_\epsilon>0$. Moreover,
 for each starting point $x_0\in U_\epsilon$, the solution $\phi(t,x_0)$
of the dynamical system $\frac{dx}{dt}(t) =H(x(t))$ satisfies $\phi(t,x_0)\in U_\epsilon$ for all $t\ge0$.
\end{lemma}

\begin{proof} The proof is completely analog, but slightly simpler than that of Lemma \ref{deterministic-boundary-B1}.
We skip the details. 
\end{proof}

Parts (2) and (3) of Lemma \ref{char-zero-D}  lead to the following explicit solution of the differential equation of
 Lemma \ref{deterministic-boundary-D}:

\begin{corollary}\label{special-solution-D}
Let $y\in C_N^D$ be the vector in Eq.~(\ref{y-max-D}). Then for
 each $c>0$, a solution of the dynamical system in  Lemma \ref{deterministic-boundary-D}
is given by $\phi(t,c\cdot y)= \sqrt{t+c^2}\cdot y $.
\end{corollary}

\begin{theorem}\label{SLLN-D}
 Let $x$ be a point in the interior of $C_N^D$, and  $y\in \mathbb R^N$. 
Let $k\ge 1/2$ with $\sqrt k \cdot x+y$ in the interior of $C_N^B$ for $k\ge k_0$. For  $k\ge k_0$, consider the Bessel processes $(X_{t,k})_{t\ge0}$ of type $D_N$ started from $\sqrt k\cdot x+y$.
Then, for all $t>0$,
\[\sup_{0\le s\le t, k\ge k_0}\|X_{s,k}- \sqrt k \phi(s,x) \|<\infty\]
 almost surely.
In particular,
\[X_{t,k}/\sqrt k\to \phi(t,x) \quad\quad\text{for}\quad\quad k\to\infty\]
locally uniformly in $t$ almost surely and thus locally uniformly in $t$ in probability.
\end{theorem}

\begin{proof} 
The proof is analog to that of  Theorem \ref{SLLN-A}. We skip the details.
\end{proof}

\begin{remark}
Let $(X_{t,k}^D)_{t\ge0}$ be a Bessel process of type D with multiplicity $k\ge0$ on the chamber $C_N^D$. Then the process
$(X_{t,k}^B)_{t\ge0}$ with 
\[X_{t,k}^{B,i}:= X_{t,k}^{D,i} \quad(i=1,\ldots,N-1), \quad X_{t,k}^{B,N}:= |X_{t,k}^{D,N} |\]
is a Bessel process of type B with the multiplicity $(k_1,k_2):=(0,k)$. This follows easily from a comparison of the
corresponding generators. 

We thus conclude from Theorem \ref{SLLN-D} that the strong LL in Theorem \ref{SLLN-B1} remains valid also for $\nu=0$. 
\end{remark}

Funding: The first author has been supported by
 the Deutsche Forschungsgemeinschaft
 (DFG) via RTG 2131 \textit{High-dimensional Phenomena in Probability - Fluctuations and Discontinuity}
 to visit Dortmund for the preparation of this paper.
 
\bibliographystyle{abbrv}
\bibliography{av-clt-bib}

\end{document}